\documentclass[11pt]{amsart}
\usepackage{amsmath}
\theoremstyle{plain}
\newtheorem{thrm}{Theorem}[section]
\newtheorem{lemma}[thrm]{Lemma}

\numberwithin{equation}{section}

\begin{document}

\title[Weingarten Equations and $\sigma_k$-Equations]
{Smooth Local Solutions\\ to Weingarten Equations and $\sigma_k$-Equations}
\author[Chen]{Tiancong Chen}
\address{Department of Mathematics\\
University of California\\
Santa Barbara, CA 93106} \email{tchen6@math.ucsb.edu}
\author[Han]{Qing Han}
\address{Department of Mathematics\\
University of Notre Dame\\
Notre Dame, IN 46556} \email{qhan@nd.edu}
\address{Beijing International Center for Mathematical Research\\
Peking University\\
Beijing, 100871, China} \email{qhan@math.pku.edu.cn}
\thanks{The second author acknowledges the support of the NSF
Grant DMS-1404596}
\maketitle
\begin{abstract}
In this paper, we study the existence of smooth local solutions 
to Weingarten equations and $\sigma_k$-equations. We will prove that, for $2\le k\le n$, 
the Weingarten equations and the $\sigma_k$-equations 
always have smooth local solutions regardless of 
the sign of the functions in the right-hand side of the equations. 
We will demonstrate that the associate linearized equations are uniformly elliptic 
if we choose the initial approximate solutions appropriately. 
\end{abstract}
\keywords{}
\subjclass{}

\section{Introduction}

Weingarten hypersurfaces are the hypersurfaces whose
principal curvatures satisfy some algebraic equations. Specifically,
let $\Omega$ be a domain in $\mathbb R^n$ and $u$ be a function defined
in $\Omega$. Suppose $\kappa_1, \cdots, \kappa_n$ are the principal curvatures of
the graph $(x, u(x))$. Then the general
Weingarten hypersurfaces are given by
\begin{equation}\label{eq-Weingarten}f(\kappa_1, \cdots, \kappa_n)=\psi(x),
\end{equation}
where $f:\mathbb R^n\to \mathbb R$ is a given function. If
$f(\kappa_1, \cdots, \kappa_n)=\kappa_1+ \cdots+ \kappa_n$, then
\eqref{eq-Weingarten} reduces to the prescribed mean curvature equation
\begin{equation}\label{eq-PrescribedMean}\Delta u-\frac{\partial_iu\partial_ju\partial_{ij}u}{1+|\nabla u|^2}
=\psi(x)\sqrt{1+|\nabla u|^2},\end{equation}
where $\psi(x)$ is the mean curvature of the graph $(x, u(x))$, 
usually denoted by $H$.
If $f(\kappa_1, \cdots, \kappa_n)=\kappa_1\cdot \cdots\cdot \kappa_n$, then
\eqref{eq-Weingarten} reduces to the  prescribed
Gauss curvature equation
\begin{equation}\label{eq-PrescribedGauss}\det(\nabla^2u)=
\psi(x)(1+|\nabla u|^2)^{\frac{n+2}2},
\end{equation}
where $\psi(x)$ is the Gauss curvature of the graph $(x, u(x))$, 
usually denoted by $K$.
It is well known that \eqref{eq-PrescribedMean} is always elliptic regardless
the sign of $\psi$ and that \eqref{eq-PrescribedGauss}
is elliptic if $\nabla^2u$ is positive definite and hence $\psi$ is positive.

Caffarelli, Nirenberg and Spruck \cite{Caf-Nir-Spr1988} studied
the Dirichlet problem for \eqref{eq-Weingarten} for a class of functions
$f$ and positive $\psi$ in strictly convex domains. The corresponding
equation is elliptic. The function $f$ in \cite{Caf-Nir-Spr1988} includes
as the special cases the $k$-th elementary symmetric functions $\sigma_k$.
We will refer the corresponding equation
\begin{equation}\label{eq-PrescribedSigma_k}
\sigma_k(\kappa_1, \cdots, \kappa_n)=\psi(x)
\end{equation}
as the prescribed $\sigma_k$-curvature equation.
We note that $k=1$ corresponds to the prescribed mean curvature equation
and that $k=n$ corresponds to the prescribed Gauss curvature equation. 
Refer to \cite{Wang2009} for the elliptic $\sigma_k$-equations and 
\cite{STW2012} for elliptic Weingarten equations. 

In this paper, we study the local solutions of
the prescribed $\sigma_k$-curvature equation \eqref{eq-PrescribedSigma_k}
and $\sigma_k$-equations.

For $n=k=2$, an equation similar to \eqref{eq-PrescribedGauss}
appears in the form of the Darboux equation
\begin{equation}
\det(\nabla^2_{g}u)=K\det(g_{ij})(1-|\nabla_gu|^2), \label{eq-Darboux}
\end{equation}
where $g$ is a smooth 2-dimensional Riemannian metric. Darboux
showed that $g$ admits a
smooth isometric embedding in $\mathbb{R}^3$ if and only if
\eqref{eq-Darboux}
admits a smooth solution $u$ with $|\nabla_gu|<1$. The equation \eqref{eq-Darboux}
is elliptic if $K$ is positive and hyperbolic if $K$ is negative.
In \cite{Lin1985} and \cite{Lin1986}, Lin proved the
existence of the sufficiently smooth isometric embedding for the
following two cases: $K(0)=0$ and $K$ nonnegative in a
neighborhood of $0\in\mathbb R^2$, or $K(0)=0$
and $dK(0)\ne 0$.
Han, Hong and Lin in \cite{HHL2003} proved the
sufficiently smooth isometric embedding if $K$ is nonpositive and
satisfies some nondegeneracy condition.
In \cite{Han2005}, Han gave an alternative proof of the result by
Lin \cite{Lin1986}.

For the general dimension, the $\sigma_n$-equation
\eqref{eq-PrescribedGauss}
exhibits a similar property as \eqref{eq-Darboux}.
Locally, the equation \eqref{eq-PrescribedGauss} can be viewed as an elliptic
equation if $\psi$ is positive and hyperbolic if $\psi$ is negative.
If $\psi(0)\neq0$, we consider a solution $u$ to \eqref{eq-PrescribedGauss} of
the following form
$$u(x)=\frac12\sum_{i=1}^{n-1}x_i^2+\frac12\operatorname{sign}\psi(0)x_n^2+w(x). $$
The linearized equation for small $w$ is a
perturbation of
$$\partial_{nn}\rho+\psi(0)\sum_{i=1}^{n-1}\partial_{ii}\rho=f.$$
This is elliptic if $\psi(0)>0$ and hyperbolic if $\psi(0)<0$. Hence we
can prove the existence of a solution to \eqref{eq-PrescribedGauss} in a
neighborhood of the origin if $\psi(0)\neq0$. If $\psi(0)=0$, the
situation is quite complicated.

Hong and Zuily in \cite{Hong-Zuily1987} 
considered \eqref{eq-PrescribedGauss} if $\psi\ge
0$. Following \cite{Lin1985}, they considered
\begin{equation}\label{1.5}u(x)=\frac12\sum_{i=1}^{n-1}x_i^2+w(x).
\end{equation} They showed that,
by adding some appropriate terms, the modified linearized
equations can be made degenerately elliptic. Based on this, they
were able to prove the existence of sufficiently smooth solutions 
\eqref{eq-PrescribedGauss}
in the general case $\psi\ge 0$ and the existence of smooth solutions
if, in addition,  $\psi$ does not vanish to infinite order or the
zero set of $\psi$ has a simple structure.

Han in \cite{Han2007} discussed \eqref{eq-PrescribedGauss} 
if $\psi$ changes sign.
He proved the existence
of sufficiently smooth solutions if $K$ changes sign cleanly, i.e.,
$\psi(0)=0$ and $\nabla \psi(0)\neq 0$. In this case, the modified linearized
equations are of the Tricomi type, elliptic in one side of the
hypersurface and hyperbolic in another side.

These results clearly demonstrate how $\psi$ determines the type of the
$\sigma_n$-equation \eqref{eq-PrescribedGauss}.
It is reasonable to expect a similar
pattern for the general $\sigma_k$-equation \eqref{eq-PrescribedSigma_k},
for $2\le k\le n-1$.
However, this is not the case. In order to find a local solution of
\eqref{eq-PrescribedSigma_k}, we can always rewrite it as a perturbation of
a linear {\it elliptic} equation if we choose the initial approximation
appropriately, regardless of the sign of $\psi$ in
\eqref{eq-PrescribedSigma_k}. Therefore, we can always find a local solution
of \eqref{eq-PrescribedSigma_k}, with no extra assumptions on $\psi$.

The main result in this paper is the following theorem.

\begin{thrm}
\label{main} Let $2\le k\le n-1$ and $\psi$ be a $C^{\infty}$-function
in a neighborhood of $0\in\mathbb R^n$. Then
\eqref{eq-PrescribedSigma_k} admits a
$C^{\infty}$-solution in some neighborhood of
$0\in\mathbb R^n$.
\end{thrm}

A similar result holds for $\sigma_k$-equations. 

Now we describe the method of proof. We consider a solution $u$
of the form
$$u(x)=\frac12\sum_{i=1}^n\mu_ix_i^2+w(x).$$
In order to have an initial approximation to an actual solution,
it is reasonable to require 
$$\sigma_k(\mu_1, \cdots, \mu_n)=\psi(0).$$
If $k=n$, this reduces to
$$\prod_{i=1}^n\mu_i=\psi(0).$$
If $\psi(0)=0$, one of $\mu_i$ has to be zero. This is the reason
that $x_n^2$ is missing in \eqref{1.5}. However, for $2\le k\le n-1$,
we can choose nonzero $\mu_i$ for all $i$ and also require the
linearized differential equation to be elliptic. In this way, the sign of
$\psi$ has no effect on the type of the equation \eqref{eq-PrescribedSigma_k}.

The paper consists of three sections including the introduction. In
Section \ref{sec-AlgebraicIneq}, we prove
some algebraic inequalities concerning the $\sigma_k$-function.
These inequalities play an important role in the proof of
Theorem \ref{main}. In Section \ref{sec-Proof}, we
use the implicit function theorem to prove Theorem \ref{main}.

\section{Algebraic Inequalities}\label{sec-AlgebraicIneq}

In the following, we set
$$\sigma_k(\mu)=\sigma_k(\mu_1, \cdots, \mu_n)
=\sum_{i_1<\cdots<i_k}\mu_{i_1}\cdots\mu_{i_k}.$$
We also set, for any $i=1, \cdots, n$,
$$\sigma_{k-1}(\mu|i)=\sigma_{k-1}(\mu_1, \cdots, \widehat \mu_i,\cdots, \mu_n),$$
where $\widehat \mu_i$ means that $\mu_i$ is deleted.

For an $n\times n$ symmetric matrix $A$, we let $\lambda(A)$ be the collection of
eigenvalues of $A$ and treat it as a vector in $\mathbb R^n$. We also write
$$\sigma_k(A)=\sigma_k\big(\lambda(A)\big).$$

\begin{lemma}\label{lemma-Expansion} Let $1\le k\le n$, $B=(b_{ij})$
be an $n\times n$ symmetric matrix
and $D=\operatorname{diag}(\mu_1, \cdots, \mu_n)$ be a diagonal
matrix, for some $\mu=(\mu_1, \cdots, \mu_n)\in\mathbb R^n$. Then,
$$\sigma_k(D+B)=\sigma_k(\mu)+\sum_{i=1}^n\sigma_{k-1}(\mu|i)b_{ii}
+O(|B|^2).$$
\end{lemma}

\begin{proof} Note that $\sigma_k(D+B)$ is a polynomial in $\mu$ of degree $k$.
The homogeneous part of degree $k$ is obviously $\sigma_k(\mu)$. We
only need to find the homogeneous part of degree $k-1$.
Recall
\begin{align*}\det(\lambda I+D+B)=\sum_{i=0}^n\sigma_{i}(D+B)\lambda^{n-i}
&=\lambda^n
+\sigma_{1}(D+B)\lambda^{n-1}+\cdots\\
&\quad +
\sigma_{k}(D+B)\lambda^{n-k}+\cdots+\sigma_{n}(D+B).\end{align*}
To find the homogeneous part in $\mu$ of degree $k-1$ in $\sigma_k(D+B)$, we
need only identify in $\det(\lambda I+D+B)$ the following term
$$\det(\lambda I+D+B)=\cdots+\sum_{p,q=1}^nc_{pq}(\mu)b_{pq}\lambda^{n-k}
+\cdots,$$
where $c_{pq}(\mu)$ is a homogeneous polynomial in $\mu$ of degree $k-1$. The
easiest way to do this is to keep $b_{ij}$ for fixed $(i,j)$ and let all
other $b_{pq}$ be zero in $\det(\lambda I+D+B)$. 
Then we expand $\det(\lambda I+D+B)$ as a polynomial
of $\lambda$ and $b_{ij}$, and identify the corresponding coefficient
of $\lambda^{n-k}$ and $b_{ij}$.
For $(i,j)$ with $i=j$, by letting all other $b_{pq}$ be zero, we have
\begin{align*}
\det(\lambda I+D+B)&=(\lambda+\mu_i+b_{ii})(\lambda+\mu_i)^{-1}
\prod_{l=1}^{n}(\lambda+\mu_l)\\
&=\cdots+\sigma_{k-1}(\mu|i)b_{ii}\lambda^{n-k}+\cdots.\end{align*}
This implies $c_{ii}(\mu)=\sigma_{k-1}(\mu|i)$. Similarly, for
$(i,j)$ with $i\neq j$, by letting all other $b_{pq}$ be zero, we have
$$
\det(\lambda I+D+B)=\big[(\lambda+\mu_i)(\lambda+\mu_j)-b_{ij}^2\big]
(\lambda+\mu_i)^{-1}(\lambda+\mu_j)^{-1}
\prod_{l=1}^{n}(\lambda+\mu_l).$$
This implies $c_{ij}=0$ for $i\neq j$.
\end{proof}

By a similar method as in the proof, we can in fact identify all terms in
$\sigma_k(D+B)$. For example, the homogeneous part in $\mu$
of degree $k-2$ is given by
$$\sum_{i<j}\sigma_{k-2}(\mu|i,j)\det\left(\begin{matrix}b_{ii}&b_{ij}\\
b_{ji}&b_{jj}\end{matrix}\right).$$

We will use Lemma \ref{lemma-Expansion} in the following way. For a $C^2$-function
$w=w(x)$, we set
$$u(x)=\frac12\sum_{i=1}^n\mu_ix_i^2+w(x).$$
We will choose constants $\mu_1, \cdots, \mu_n$ 
and a sufficiently small function $w$ appropriately such that $u$ defined above is the desired solution. 
To this end, we need to analyze the linearization of $\sigma_k(\nabla^2u)$ with respect to $w$
at $w=0$. 
Note $$\nabla^2u=D+\nabla^2w,$$
where $D=\operatorname{diag}(\mu_1, \cdots, \mu_n)$. 
Hence,  the linearization of $\sigma_k(\nabla^2u)$ with respect to $w$
at $w=0$ is given by
$$Lv=\sum_{i=1}^n\sigma_{k-1}(\mu|i)\partial_{ii}v.$$
We now demonstrate that we can always make this operator elliptic
by choosing $\mu_1, \cdots, \mu_n$ appropriately.

\begin{lemma}\label{lemma-PositiveCoefficients}
Let $2\le k\le n-1$. For any constant $M\in \mathbb R$, there exists a vector $\mu
=(\mu_1, \cdots, \mu_n)\in \mathbb R^n$ such that
\begin{equation}\label{eq-AlgebraicSigma_k}\sigma_k(\mu)=M\end{equation}
and
\begin{equation}\label{eq-PositiveSigma_k-1}
\sigma_{k-1}(\mu|i)>0\quad\text{for any }i=1,\cdots, n.\end{equation}
\end{lemma}

\begin{proof} We write \eqref{eq-AlgebraicSigma_k} as
$$\sigma_{k}(\mu|n)+\sigma_{k-1}(\mu|n)\mu_{n}=M.$$
We assume
$$\sigma_{k-1}(\mu|n)>0.$$
Then we have
\begin{equation}\label{eq-11}
\mu_{n}=\frac{1}{\sigma_{k-1}(\mu|n)}\big[M-\sigma_{k}(\mu|n)\big].\end{equation}
Next, we fix an $i=1, \cdots, n-1$ and write
$$\sigma_{k-1}(\mu|i)=\sigma_{k-1}(\mu|i,n)+\sigma_{k-2}(\mu|i,n)\mu_n,$$
where, for $l=k-1, k-2$,
$$\sigma_{l}(\mu|i,n)=\sigma_{l}(\mu_1, \cdots, \widehat \mu_i, \cdots,
\mu_{n-1}, \widehat \mu_n).$$
Hence, \begin{align*}
\sigma_{k-1}(\mu|i)&=\sigma_{k-1}(\mu|i,n)+
\frac{\sigma_{k-2}(\mu|i,n)}{\sigma_{k-1}(\mu|n)}\big[M-\sigma_{k}(\mu|n)\big] \\
&=\frac{1}{\sigma_{k-1}(\mu|n)}\big[\sigma_{k-1}(\mu|n)\sigma_{k-1}(\mu|i,n)
-\sigma_{k}(\mu|n)\sigma_{k-2}(\mu|i,n)+\sigma_{k-2}(\mu|i,n)M\big]\\
&=\frac{1}{\sigma_{k-1}(\mu|n)}\bigg\{\sigma_{k-1}(\mu|i,n)
\big[\sigma_{k-1}(\mu|i,n)+\sigma_{k-2}(\mu|i,n)\mu_i\big]\\
&\qquad\qquad\qquad-\sigma_{k-2}(\mu|i,n)
\big[\sigma_{k}(\mu|i,n)+\sigma_{k-1}(\mu|i,n)\mu_i\big]+\sigma_{k-2}(\mu|i,n)M
\bigg\}\\
&=\frac{1}{\sigma_{k-1}(\mu|n)}\bigg\{\big[\sigma_{k-1}(\mu|i,n)\big]^2
-\sigma_{k-2}(\mu|i,n)
\sigma_{k}(\mu|i,n)+\sigma_{k-2}(\mu|i,n)M
\bigg\}. \end{align*}
By Newton's inequality (in $\mathbb R^{n-2}$), we have
$$\sigma_{k-2}(\mu|i,n)
\sigma_{k}(\mu|i,n)\le \frac{(k-1)(n-k-1)}{k(n-k)}
\big[\sigma_{k-1}(\mu|i,n)\big]^2,$$
and hence
\begin{align}\label{eq-12}\begin{split}
\sigma_{k-1}(\mu|i)\ge \frac{1}{\sigma_{k-1}(\mu|n)}
\bigg\{\left[1-\frac{(k-1)(n-k-1)}{k(n-k)}\right]
\big[\sigma_{k-1}(\mu|i,n)\big]^2&\\
+\sigma_{k-2}(\mu|i,n)M&\bigg\}.\end{split}\end{align}
In the following, we set
$$1_{\mathbb R^n}=(1,\cdots,1)\in\mathbb R^n.$$
Take $\mu_1=\cdots=\mu_{n-1}=N$, with $N$ to be determined. Then,
$$\sigma_{k}(\mu|n)=\sigma_{k}(1_{\mathbb R^{n-1}})N^{k},$$
and
$$\sigma_{k-1}(\mu|n)=\sigma_{k-1}(1_{\mathbb R^{n-1}})N^{k-1}>0.$$
By \eqref{eq-11}, we take
$$\mu_{n}=\frac{1}{\sigma_{k-1}(1_{\mathbb R^{n-1}})N^{k-1}}
\big[M-\sigma_{k}(1_{\mathbb R^{n-1}})N^{k}\big].$$
Then by \eqref{eq-12}, we have
\begin{align*}
\sigma_{k-1}(\mu|i)\ge \frac{1}{\sigma_{k-1}(1_{\mathbb R^{n-1}})N^{k-1}}
\bigg\{\left[1-\frac{(k-1)(n-k-1)}{k(n-k)}\right]
\cdot\big[\sigma_{k-1}(1_{\mathbb R^{n-2}})N^{k-1}\big]^2&\\
+\sigma_{k-2}(1_{\mathbb R^{n-2}})N^{k-2}M&\bigg\}.\end{align*}
For the given $M$, we can choose $N$ sufficiently large such that
$$\sigma_{k-1}(\mu|i)\ge cN^{k-1}>0,$$
where $c$ is a positive constant depending only on $n$ and $k$. 
This establishes the desired result. \end{proof}

To prove Theorem \ref{main}, we will construct solutions $u$ 
as perturbations of the initial approximation 
\begin{equation}\label{eq-InitialApproximation}u_0=\frac12\sum_{i=1}^{n}
\mu_ix_i^2,\end{equation}
where $\mu_1, \cdots, \mu_n$ are from Lemma \ref{lemma-PositiveCoefficients}. 
Hence, local behaviors of $u$ such as the convexity coincide with those of $u_0$ 
in \eqref{eq-InitialApproximation}, 
if none of $\mu_i$ is zero.  The proof of Lemma \ref{lemma-PositiveCoefficients}
gives one choice of $\mu_1, \cdots, \mu_n$ satisfying 
\eqref{eq-AlgebraicSigma_k} and \eqref{eq-PositiveSigma_k-1}, which may not result in 
solutions $u$ with good geometric properties. In some cases, 
better choices of $\mu_1, \cdots, \mu_n$ are available.  For example, if $M>0$, we can 
choose $\mu_1=\cdots=\mu_n=N$ for some appropriate $N>0$. Then,
the corresponding $u_0$ is convex and so will be the resulting solution $u$.

\section{Proof of the Main Theorem}\label{sec-Proof}

In the present section, we discuss the linearization of the
$\sigma_k$-equation and prove Theorem \ref{main}.

Let $u$ be a $C^2$-function defined in a domain in
$\mathbb R^n$. According to \cite{Caf-Nir-Spr1986},
the principal curvatures $\kappa$ of the graph of $u$ are eigenvalues of
the symmetric matrix
\begin{equation}\label{eq-PrincipalCurvSetup}
a_{ij}=\frac{1}{v}\left(u_{ij}-\frac{u_iu_ku_{ki}}{v(1+v)}
-\frac{u_ju_lu_{lj}}{v(1+v)}+\frac{u_iu_ju_ku_lu_{kl}}{v^2(1+v)^2}\right),
\end{equation}
where $v=(1+|\nabla u|^2)^{1/2}$ and the summation convention was used.

Let $\psi$ be a $C^\infty$-function defined in a neighborhood
of $0\in\mathbb R^n$. By setting $M=\psi(0)$, 
we take $\mu_1, \cdots,\mu_n$ satisfying Lemma 
\ref{lemma-PositiveCoefficients}, i.e.,
\begin{equation}\label{eq-AlgebraicSigma_kSetup}\sigma_k(\mu)=
\psi(0),\end{equation}
and
\begin{equation}\label{eq-PositiveSigma_k-1Setup}
\sigma_{k-1}(\mu|i)>0\quad\text{for any }i=1,\cdots, n.\end{equation}

To proceed, we temporarily replace $x \in \mathbb{R}^n $ by
$\tilde{x}\in \mathbb{R}^n $ and write
$\widetilde{\partial_i}$ instead of $\partial_{\tilde{x}_i}
$. Consider
\begin{equation}\label{2.5}
\tilde {\mathcal{F}}(u)=\sigma_k(\tilde \kappa_1, \cdots,
\tilde\kappa_n)-\psi(\tilde x),
\end{equation}
where $\tilde \kappa=(\tilde \kappa_1, \cdots,
\tilde\kappa_n)$ is the collection of all principal curvatures of
the graph $(\tilde x, u(\tilde x))$.
Now  we set, for $\varepsilon>0$,
$$\tilde{x}=\varepsilon^2 x, $$
and
\begin{equation*}\label{2.6}u(\tilde{x})=\frac12\sum_{i=1}^{n}
\mu_i\tilde x_i^2+\varepsilon^{5}w\left(\frac{\tilde{x}}
{\varepsilon^2} \right).\end{equation*}
We evaluate $\tilde {\mathcal{F}}(u)$ in terms of $w$ by setting
\begin{equation}\label{2.8}
\mathcal{F}(w,
\varepsilon)=\frac{1}{\varepsilon}\tilde
{\mathcal{F}}(u).\end{equation}

We first note that
$$\tilde\partial_iu=\mu_l\tilde x_l\delta_{il}+\varepsilon^3 \partial_iw=
\varepsilon^2\mu_i x_i+\varepsilon^3 \partial_iw,$$
and
$$\tilde\partial_{ij}u=\mu_l\delta_{il}\delta_{jl}+\varepsilon \partial_{ij}w.$$
Then the matrix $A=(a_{ij})$ in \eqref{eq-PrincipalCurvSetup} has the form
$$
a_{ij}=\mu_{i}\delta_{ij}+\varepsilon \partial_{ij}w
+\varepsilon^2b_{ij}(\nabla w, \nabla^2w, \varepsilon),$$
where $b_{ij}$ is a smooth function in its arguments.
By Lemma \ref{lemma-Expansion}, we have
$$\sigma_k(A)=\sigma_k(\mu)+\varepsilon
\sum_{i=1}^n\sigma_{k-1}(\mu|i)\partial_{ii}w
+\varepsilon^2 c(\nabla w, \nabla^2w, \varepsilon),$$
where $c$ is a smooth function in its arguments.
By substituting in \eqref{2.5}, we have 
\begin{equation*}
\tilde {\mathcal{F}}(u)=\varepsilon
\sum_{i=1}^n\sigma_{k-1}(\mu|i)\partial_{ii}w
+\varepsilon^2 c(\nabla w, \nabla^2w, \varepsilon)
-\big(\psi(\varepsilon^2 x)-\psi(0)\big).
\end{equation*} 
By the mean value theorem (applied to $\psi$) and \eqref{2.8}, we obtain 
\begin{equation}\label{2.8a}
\mathcal{F}(w,
\varepsilon)=\sum_{i=1}^n\sigma_{k-1}(\mu|i)\partial_{ii}w
+\varepsilon f(x, \nabla w, \nabla^2w, \varepsilon),\end{equation}
where $f$ is a smooth function in its arguments. Our goal is to solve 
$\mathcal{F}(w,\varepsilon)=0$ for sufficiently small $\varepsilon$. 

\begin{proof}[Proof of Theorem \ref{main}] Let $\mathcal F$ be as given in 
\eqref{2.8a}. Then 
\begin{equation*}
\mathcal{F}(w,
0)=\sum_{i=1}^n\sigma_{k-1}(\mu|i)\partial_{ii}w,\end{equation*}
and hence $\mathcal{F}(0,0)=0$. Let $L$ be the linearized operator of $\mathcal F$ 
with respect to $w$ at $(w,\varepsilon)=(0,0)$. Then 
$$Lv=\sum_{i=1}^n\sigma_{k-1}(\mu|i)\partial_{ii}v.$$ 
By \eqref{eq-PositiveSigma_k-1Setup}, 
$L$ is an elliptic operator of constant coefficients. For fixed
$m\ge 3$ and $\alpha\in (0,1)$, we set 
$$X=\{u\in C^{m,\alpha}(\bar B_1):\, u=0\text{ on }\partial B_1\}, \qquad 
Y=C^{m-2,\alpha}(\bar B_1).$$
We note that 
$$\mathcal F: X\times (-\varepsilon_0, \varepsilon_0)\to Y$$ 
is a well-defined map for sufficiently small $\varepsilon_0$ and that 
$L: X\to Y$ is an isomorphism by the classical Schauder theory. Then by the 
implicit function theorem, for sufficiently small $\varepsilon$, there 
exists a $w=w(\varepsilon)\in X$ such that $\mathcal F(w,\varepsilon)=0$. 
Note that $\mathcal F$ is a fully nonlinear elliptic operator for sufficiently small 
$\varepsilon$. The classical Schauder theory implies that $w\in C^\infty(\bar B_1)$. 
\end{proof}

\end{document}